\documentclass[11pt]{amsart}
\usepackage{amsmath, amssymb}
\ifx\pdfoutput\undefined 
\usepackage{graphicx}
\else
\usepackage[pdftex]{graphicx}
\usepackage{epstopdf}
\fi
\usepackage{verbatim, manfnt, amscd}
\usepackage{tikz}

\newtheorem{theorem}{Theorem}[section]
\newtheorem{proposition}[theorem]{Proposition}
\newtheorem{lemma}[theorem]{Lemma}

\theoremstyle{definition}

\theoremstyle{remark}
\newtheorem{remark}[theorem]{Remark}
\newtheorem{remarks}[theorem]{Remarks}
 
\numberwithin{equation}{section}

\newcommand{\bbC}{{\mathbb C}}

\newcommand{\bbZ}{{\mathbb Z}}
\newcommand{\bbR}{{\mathbb R}}

\newcommand{\bbP}{{\mathbb P}}
\newcommand{\bbT}{{\mathbb T}}

\newcommand{\calR}{{\mathcal R}}

\newcommand{\bel}{\begin{equation}\label}
\newcommand{\ee}{\end{equation}}

\begin{document}

\openup1pt

\title[The Generalized Legendre transform and its applications]{The Generalized Legendre transform and its applications to inverse spectral problems}

\author{Victor Guillemin}\address{Department of Mathematics, Massachusetts Institute of Technology, Cambridge, MA 02139, USA}
\thanks{Victor Guillemin is supported in part by NSF grant
DMS-1005696.}
\email{vwg@math.mit.edu} 
\author{Zuoqin Wang}\address{School of Mathematical Sciences, University of Science and Technology of China, Hefei, Anhui 230026, P.R.China}\email{wangzuoq@ustc.edu.cn}
 
\begin{abstract}
Let $M$ be a Riemannian manifold, $\tau: G \times M \to M$ an isometric action on $M$ of an $n$-torus $G$ and $V: M \to \mathbb R$ a bounded $G$-invariant smooth function. By $G$-invariance the Schr\"odinger operator, $P=-\hbar^2 \Delta_M+V$, restricts to a self-adjoint operator on $L^2(M)_{\alpha/\hbar}$, $\alpha$ being a weight of $G$ and $1/\hbar$ a large positive integer. Let $[c_\alpha, \infty)$ be the asymptotic support of the spectrum of this operator. We will show that $c_\alpha$ extend to a function, $W: \mathfrak g^* \to \mathbb R$ and that, modulo assumptions on $\tau$ and $V$ one can recover $V$ from $W$, i.e. prove that $V$ is spectrally determined. The main ingredient in the proof of this result is the existence of a ``generalized Legendre transform" mapping the graph of $dW$ onto the graph of $dV$. 
\end{abstract}
 
\maketitle

\section{Introduction}

Let $G$ be an $n$ dimensional torus and $\mathfrak g$ the Lie algebra of $G$. Given a weight $\alpha \in \mathfrak g^*$, we will denote by $\chi_{\frac{\alpha}{\hbar}}: G \to S^1$ the character of $G$ associated with the weight $\frac{\alpha}{\hbar}$, $\frac 1{\hbar}$ being a large integer. 

Now let $M$ be a Riemannian manifold, $V: M \to \bbR$ a $C^\infty$ function and 
\begin{equation}\label{P}
P = \hbar^2 \Delta_M +V
\end{equation}
the semi-classical Schr\"odinger operator associated with $V$, where $\Delta_M$ is the Laplace-Beltrami operator acting on $L^2(M)$. In order to ensure that the spectrum of $P$ is discrete we will assume that $M$ is compact, or, if not, that $V$ is proper and tends to $+\infty$ as $x$ tends to infinity in $M$. 

Now let $\tau: G \times M \to M$ be an isometric action of $G$ on $M$ and assume that the function $V$ alluded above is an element of $C^\infty(M)^G$. We will denote by $L^2(M)_{\frac{\alpha}{\hbar}}$ the space of $L^2$ functions on $M$ which transform under the action of $G$ by the character $\chi_{\alpha/\hbar}$, i.e.   
\begin{equation}\label{ch}
L^2(M)_{\frac{\alpha}{\hbar}} = \left\{ f \in L^2(M)\ |\ \gamma_g^* f = \chi_{\frac{\alpha}{\hbar}}(g) f\right\},
\end{equation}
and we will denote by $\mu_{{\alpha}, {\hbar}}$ the spectral measure of the operator, $P$, restricted to $L^2(M)_{\frac{\alpha}{\hbar}}$, i.e. for $\rho \in C_0^\infty(\bbR)$, 
\begin{equation}
\label{spec}
 \mu_{\alpha, \hbar} (\rho )  = \sum \rho (\lambda_i(\alpha, \hbar)),
\end{equation}
where 
\[
\lambda_i(\alpha, \hbar), \quad i=1, 2, 3, \cdots
\]
are the eigenvalues of $P$ restricted to $L^2(M)_{\frac{\alpha}{\hbar}}$. The asymptotic behavior of such equivariant eigenvalues has been studied by various authors, see e.g. \cite{BH}, \cite{Do} and \cite{GU}. Recently in the article \cite{DGS} E. Dryden, V. Guillemin and R. Sena-Dias obtained an asymptotic formula for the measure (\ref{spec}) in terms of the symplectic reduction $(T^*M)_\alpha$ of $T^*M$. More precisely, if $\tilde\tau$ is the lifted $G$-action on $T^*M$, then $\tilde\tau$ is a Hamiltonian action and we will denote its moment map by $\phi: T^*M \to \mathfrak g^*$. One can show (c.f. \cite{GS1}) that if $\alpha$ is a regular value of $\phi$, then $G$ acts in a locally free fashion on $\phi^{-1}(\alpha)$. For simplicity let's assume that this action is free, so that  the quotient 
\begin{equation}\label{sympred}
(T^*M)_\alpha = \phi^{-1}(\alpha) / G 
\end{equation}
is a smooth manifold which inherits from $T^*M$ a quotient symplectic structure. 

Now let $p: T^*M \to \bbR$ be the semi-classical symbol of the operator (\ref{P}), i.e. at $\xi \in T_m^*M$ let 
\[
p(m, \xi) = (\xi, \xi)_m^* +V(m),
\]
where $(\cdot, \cdot)_m^*$ is the  inner product on $T_m^*M$ induced by the Riemannian metric on $M$. Since this symbol is $G$-invariant, its restriction to $\phi^{-1}(\alpha)$ is the pull-back to $\phi^{-1}(\alpha)$ of a $C^\infty$ function 
\begin{equation}
\label{palpha}
p_\alpha: (T^*M)_\alpha \to \bbR
\end{equation}
and the asymptotic formula for $\mu_{\alpha, \hbar}$ that we alluded to above asserts that as $\hbar \to 0$,
\begin{equation}\label{mualphah}
\mu_{\alpha, \hbar} = (2\pi \hbar)^{n-d} \Big( (p_\alpha)_* \nu_\alpha + O(\hbar) \Big),
\end{equation}
where $\nu_\alpha$ is the symplectic volume form on $(T^*M)_\alpha$ and $d$ the dimension of $M$. 

As was observed by Abraham-Marsden in \cite{AM}, \S 4.3, the reduced space (\ref{sympred}) and the map (\ref{palpha}) have the following alternative description:  Let $M_0$ be the open submanifold of $X$ on which $G$ acts freely, let $X = M_0/G$ and let $\pi: M_0 \to X$ be the fibration of $M_0$ over $X$. Then as a manifold
\begin{equation}
\label{TY}
(T^*M)_\alpha \simeq T^*X
\end{equation}   
and under this identification, $p_\alpha: T^*X \to \bbR$ is the symbol of the Schr\"odinger operator 
\begin{equation}
\label{Sc}
\hbar^2 \Delta_X + V_\alpha,
\end{equation} 
where $\Delta_X$ is the Laplace operator associated with the quotient Riemann metric on $X$ and $V_\alpha$ is the potential function on $X$ defined by
\begin{equation}\label{V}
(\pi^* V_\alpha)(m) = V(m) + \langle \alpha, \alpha\rangle_m^*.
\end{equation}
(To make sense of the second summand note that one gets from the free action of $G$ on $M_0$ an identification of $\mathfrak g$ with the vertical tangent space of $M_0$ at $m$, and hence from the Riemannian inner product, $(\cdot, \cdot)_m$ on $T_mM_0$ an inner product, $\langle \cdot, \cdot \rangle_m$ on $\mathfrak g$ and a dual inner product, $\langle \cdot, \cdot \rangle_m^*$ on $\mathfrak g^*$. )

\begin{remarks}
$\mathrm{(1)}$ The canonical symplectic form on $T^*X$ does not, in general, coincide with the reduced symplectic form on $(T^*M)_\alpha$. However one can show that the symplectic volume forms do, and hence the measure $(p_\alpha)_* \nu_\alpha$ in the asymptotic formula (\ref{mualphah})  \emph{is} just the push-forward by the functions, $p_\alpha$, of the symplectic volume form on $T^*X$. 

\noindent $\mathrm{(2)}$ In most cases of interest, the manifold $M_0$ is \emph{not} compact. However, one can show that the function $p_\alpha: T^*X \to \bbR$ is proper and hence that $(p_\alpha)_*\nu_\alpha$ is, in all cases, well defined. 
\end{remarks}

Let $[c_\alpha, +\infty)$ be the support of the measure $(p_\alpha)_*\nu_\alpha$. Then by (\ref{V}), the quantity 
\begin{equation}
\label{calpha}
c_\alpha = \min_{x \in X} V_\alpha(x)
\end{equation}
is a spectral invariant of the operator (\ref{P}). We remark that although at first glance this invariant is only defined for integer weights, i.e. weights $\alpha \in \mathfrak g^*$ that sits in the weight lattice, it is actually defined for all rational weights $\alpha \in \mathfrak g^*$ since $\alpha/\hbar$ is in the weight lattice for some small $\hbar$, and thus by continuation defined for all $\alpha \in \mathfrak g^*$.  Our goal in this paper will be to show that in a number of interesting cases this invariant determines the potential $V$. One simple example of a result of this type was proved in \cite{DGS} where it was shown that in the case of  the Schr\"odinger operator 
\begin{equation}
\label{RSch}
\hbar^2 \Delta_{\bbR^{2n}} + V(y_1^2+y_2^2, \cdots, y_{2n-1}^2+y_{2n}^2)
\end{equation}
and the group 
\begin{equation}\label{torus}
G = SO(2) \times \cdots \times SO(2), \qquad (n \mbox{\ factors}) 
\end{equation}
the equivariant spectrum of (\ref{RSch}) determines the function $V=V(r_1,\cdots, r_n)$ providing we impose the following growth conditions on $V$:
\begin{subequations}\label{DGS-C}
\begin{align}
& \mathrm{(a)}\ V \mbox{\ is proper on} \mathbb R^n_+,\label{DGS-Ca}\\
& \mathrm{(b)}\ \frac{\partial V}{\partial r_i}>0  \mbox{ for all } i, \label{DGS-Cb}\qquad\qquad\qquad\qquad \\
& \mathrm{(c)}\ \left[\frac{\partial^2 V}{\partial r_i\partial r_j}\right] \ge 0. \label{DGS-Cc}
\end{align}
\end{subequations}
 More explicitly, in this example we have, for $x_i^2 = y_{2i-1}^2+y_{2i}^2$, 
\begin{equation}
\label{Vralpha}
V_\alpha(x) := V(x_1^2, \cdots, x_n^2) + \sum \frac{\alpha_i^2}{x_i^2}.
\end{equation}
Therefore if we let 
\[  
s_i  =\alpha_i^2, \quad
 c_\alpha  = G(s_1, \cdots, s_n), \quad
t_i  = \frac 1{x_i^2}
\] 
and 
\[F(t_1, \cdots, t_n) =-V(\frac 1{t_1}, \cdots, \frac 1{t_n}),\] 
then the formula (\ref{calpha}) can be rewritten as 
\[
G(s) = \min_{t \in \mathbb R^n_+} [t \cdot s - F(t)]. 
\]
The minimum is achieved at the unique critical point $s= \frac{\partial F}{\partial t}$.  Moreover, under the hypotheses above on $V$, $\frac{\partial F}{\partial t}$ is a bijection 
\[
\frac{\partial F}{\partial t}: \bbR_+^n \to \bbR_+^n.
\]
Therefore by the identity 
\[
t \cdot s-F(t)=G(s) \qquad \mbox{\ for \ } s=\frac{\partial F}{\partial t}
\]
and the inversion formula for the Legendre transform 
\begin{equation}\label{Leg}
s=\frac{\partial F}{\partial t} \Longleftrightarrow t = \frac{\partial G}{\partial s}
\end{equation}
we get 
\begin{equation}\label{FW}
F(t) = s \frac{\partial G}{\partial s} - G(s) \qquad \mbox{for\ } t=\frac{\partial G}{\partial s}.
\end{equation}
Therefore, since $c_\alpha = G(\alpha_1^2, \cdots, \alpha_n^2)$, $G$ is spectrally determined and hence so are the functions $F$ and $V$. 

We will show in this paper that this ``Legendre transform" argument can be considerably generalized. More precisely, let 
\[W(x, \alpha)=\langle \alpha, \alpha \rangle^*_{\pi^{-1}(x)}\]
be the function appearing in (\ref{V}) and let 
\[
V(x, \alpha)=V(\pi^{-1}(x))+W(x, \alpha).
\]
Suppose  $  W(x, \alpha) $ is the generating  function of a canonical transformation
\begin{equation}
\tag{1.17a}\label{a}
 \gamma: T^*X \to T^*\bbR^n
\end{equation}
and in addition for all $ \alpha \in \mathbb{R}^n $
\begin{equation}
\tag{1.17b}\label{b}
V(x, \alpha) 
\end{equation}
\noindent has a unique non-degenerate minimum $ G(\alpha). $

We will then show that the potential $V(x)$ can be reconstructed from the function $G$, and hence that $V$ can be reconstructed from the equivariant spectral invariants $c_\alpha$. This result will be proved in section 3 and in the following sections we will discuss applications of it. More explicitly, in section 3.1 we will use this result to get an improved version of the inverse result that we described above for the Schr\"odinger operator (\ref{RSch}) (namely we will show that the hypothesis (\ref{DGS-Cb}) is unnecessary), and we will also show that the proof of this result, with small modifications, gives one an inverse result for the Schr\"odinger operator on the $n$-fold product $\mathbb{CP}^1 \times \cdots \times \mathbb{CP}^1$. Then in section 4 we will focus on a problem of which this example is a special case: the inverse spectral problem for the Schr\"odinger operator on an arbitrary toric variety, $M$, and state some sufficient conditions on potentials, $V$, and subregions, $Y$, of $X$ that will guarantee spectral determinability of the restriction of $V$ to $Y$. (These conditions are in general rather hard to verify in practice, but we will discuss their verifiability for two interesting classes of toric varieties: complex projective spaces and Hirzebruch surfaces)

In section 5 we will turn our attention to some inverse spectral problems having to do with the ``local spectral invertibility" and ``local spectral rigidity" of a Schr\"odinger potential, $V$, in the neighborhood of a non-degenerate minimum of $V_\alpha$. (The proof of these results is based on the fact that if the critical points, $p$, of $V_\alpha$ are non-degenerate and the corresponding critical values are distinct, then these critical values are also equivariant spectral invariants of the Schr\"odinger operator. Hence one can make use of the Legendre techniques of section 3 to obtain inverse results for perturbations of $V$ in small neighborhood of these points.) 
 
Finally in the last section of this paper we will show how one can deduce from the inverse spectral results described above for the semi-classical Schr\"odinger operator similar results for the classical Laplace operator on a line bundle using the method of ``reduction in stages" described in \cite{DGS} section 3. (We would also like, by the way, to thank two of the co-authors of this paper, Emily Dryden and Rosa Sena-Dias, for a number of very helpful suggestions bearing on the results we've just described.)


\section{Toric manifolds and their canonical reduced Riemannian metrics} $ $

Let $M$ be a compact  toric manifold, i.e. a $2n$ dimensional symplectic manifold which admits an effective Hamiltonian action of $G= \bbT^n$.  Let $\phi: M \to \mathfrak g^* = \bbR^n$ be its moment map. Then it is well known that the image $\overline{\mathcal P}$ of $\phi$ is a \emph{Delzant polytope}, i.e. a convex polytope  in $\bbR^n$ such that 
\begin{itemize}
\item (\emph{simplicity}) there are $n$ edges meeting at each vertex;
\item (\emph{rationality}) the edges meeting at a vertex $p$ are of the form 
\[
p +tu_i \quad (t \ge 0)
\] 
with $u_i \in \bbZ^n, 1\le i \le n$;
\item (\emph{smoothness}) the vectors $u_1, \cdots, u_n$ can be chosen as a $\bbZ$-basis of $\bbZ^n$.  
\end{itemize}
A classical result of Delzant (\cite{De}) asserts that there is a one-to-one correspondence between compact connected $2n$ dimensional toric manifolds and Delzant polytopes in $\bbR^n$, and in fact, we now know that \emph{most} geometric and topological information about $M$ is encoded in  the combinatoric  of $\overline{\mathcal P}$, see, e.g. the books \cite{Au} and \cite{G3}.  

According to Delzant's construction,  $M$ admits an intrinsic $G$-invariant complex structure which is compatible with its symplectic form, and hence gives $M$ an intrinsic K\"ahler, and thus Riemannian metric. 
We will denote by ${\mathcal P}$ the interior of $\overline{\mathcal P}$. It is well known that the dense open subset 
\[
M_0 = \phi^{-1}({\mathcal P})
\] 
of $M$ is a complex torus with respect to the intrinsic complex structure, and the $G$-action is free precisely on $M_0$.  

Before we apply the result in the previous section to this setting, we first observe that the ``base manifold'' $X=M_0/G$ is exactly the open polytope ${\mathcal{P}}$. In particular, the symplectic quotient $(T^*M)_\alpha$ is identified with $T^*{\mathcal{P}}$. 

Now let $V \in C^\infty(M)^G$ be a $G$-invariant smooth function on $M$. Then the restriction of $V$ to $M_0$ gives rise to a smooth function, still denoted by $V$, on ${\mathcal{P}}$.  Similarly for any weight $\alpha$, $\langle \alpha, \alpha\rangle^*$ descends to a function on ${\mathcal{P}}$. So the result alluded in the previous section becomes 
\begin{theorem}\label{toricm}
The equivariant spectrum of the Schr\"odinger operator $P=\hbar^2 \Delta_M +V$ determines the function 
\begin{equation}\label{malpha}
G(\alpha) = \min_{x \in {\mathcal{P}}} \Big( V(x)+\langle \alpha, \alpha \rangle^* (x) \Big).
\end{equation}
\end{theorem}

It remains to describe the function $\langle \alpha, \alpha \rangle^* (x)$ explicitly.  By definition, $\overline{\mathcal P}$ is defined by a set of inequalities of the form 
\begin{equation}\label{ineq}
l_i(x) = \sum_{j=1}^n l_i^j x_j + l_i^0 \ge 0, \quad 1 \le i \le d,
\end{equation}
where $\vec l_i = \langle l_i^1, \cdots, l_i^n \rangle^T$ is a primitive element of the lattice $\bbZ^n$, and $d$ is the number of facets of $\overline{\mathcal P}$.  Note that $x \in {\mathcal P}$ if and only if $l_i(x)>0$ for all $i$. 
It was proven by one of the authors  (\cite{G2}) that the induced Riemannian metric on ${\mathcal P}$ is precisely given by the formula
\begin{equation}\label{metric}
\frac 12 \sum_{i=1}^d \frac{(dl_i)^2}{l_i(x)}
\end{equation}
from which he derived an explicit formula for the canonical K\"ahler metric on $M$ and from it the formula
\begin{equation}\label{Walphax}
W(x, \alpha): = \langle \alpha, \alpha \rangle^*(x) = \frac 12 \sum_{i=1}^{d} \frac{(dl_i(\alpha))^2}{l_i(x)}. 
\end{equation}

\section{The inversion formula for the Legendre transform revisited}

Let $\mathcal U_1$ and $\mathcal U_2$ be connected open subsets of $\mathbb R^n$, 
\[\gamma: T^*\mathcal U_1 \to T^* \mathcal U_2\]
a canonical transformation and $\Lambda_1$ a Lagrangian submanifold of $T^*\mathcal U_1$. Then $\Lambda_2 = \gamma(\Lambda_1)$ is a Lagrangian submanifold of $T^* \mathcal U_2$ and 
\[
\Gamma = \{ (x, y, -\xi, \eta)\ |\ (y, \eta) = \gamma(x, \xi)\}
\]
a Lagrangian submanifold of $T^*(\mathcal U_1 \times \mathcal U_2)$. 
 
Suppose that these Lagrangian manifolds are \emph{horizontal}, i.e. suppose that there exist functions $F \in C^\infty(\mathcal U_1), G \in C^\infty(\mathcal U_2)$ and $W \in C^\infty(\mathcal U_1 \times \mathcal U_2)$ such that 
\begin{subequations}\label{horiLag}
\begin{align}
& \mathrm{(a)}\  \quad (x, \xi) \in \Lambda_1 \Longleftrightarrow \xi = \frac{\partial F}{\partial x}, \label{horiLag-a}\\
& \mathrm{(b)}\ \quad (y, \eta) \in  \Lambda_2 \Longleftrightarrow \eta = \frac{\partial G}{\partial y},
\label{horiLag-b}\\ 
& \mathrm{(c)}\ \quad (x, y, \xi, \eta) \in \Gamma \Longleftrightarrow \xi = \frac{\partial W}{\partial x}, \eta = \frac{\partial W}{\partial y}.\label{horiLag-c}
\end{align}
\end{subequations}
Then for $(x, \xi) \in \Lambda_1$, $(y, \eta) = \gamma(x, \xi)$ if and only if
\begin{equation}
\frac{\partial F}{\partial x}(x) =- \frac{\partial W}{\partial x} (x, y) \quad \mathrm{and} \qquad
\frac{\partial G}{\partial y}(y) = \frac{\partial W}{\partial y}(x, y).
\end{equation}

Let's now make the additional assumption that for all $y$, the function 
\[
x \mapsto F(x)+W(x, y)
\]
has a unique critical point and that this point is a global minimum. 
\begin{theorem}
Under this assumption, 
\begin{equation}
\label{GFW}
G(y) = \min_x (F(x)+W(x, y))+C,
\end{equation}
$C$ being an additive constant (which we can set equal to zero by replacing $G$ by $G-C$. )
\end{theorem}
\begin{proof}
$\Lambda_1$ and $\Lambda_2$ are defined by the equations 
\[
\xi = \frac{\partial F}{\partial x} \qquad \mathrm{and}\qquad \eta = \frac{\partial G}{\partial y},
\]
and the graph of $\gamma$ by the equations 
\[
\xi = -\frac{\partial W}{\partial x} \qquad \mathrm{and}\qquad \eta =\frac{\partial W}{\partial y}.
\]
Thus if $\sigma: \mathcal U_2 \to \mathcal U_1$ is the map 
\[
y \in \mathcal U_2 \longrightarrow \left(y, \frac{\partial G}{\partial y}(y)\right) \stackrel{\gamma^{-1}}{\longrightarrow} \left(x, \frac{\partial F}{\partial x}(x)\right) \to x,
\]
the restriction of the the one-form $d(F+W-G)$ to the graph of $\sigma$ vanishes. Hence on this graph, $F+W-G = C$, and for $x = \sigma(y)$, 
\[
\frac{\partial}{\partial x} \left(F(x)+W(x, y)\right) = 0.
\]
In other words, $x$ is the unique global minimum of the function $x \mapsto F(x)+W(x, y)$, and at this minimum, 
\[
G(y) = F(x)+W(x, y). 
\]
\end{proof} 
\begin{remark}
The classical inversion theorem for Legendre transforms is easily derived from this result: Recall that a function $F \in C^\infty(\bbR^n)$ is \emph{strictly convex} if satisfies 
\[
\left[ \frac{\partial^2 F}{\partial x_i\partial x_j} \right] >0
\]
for all $x$. Suppose for simplicity that $F$ is \emph{stable}, i.e. $F$ has a unique critical point which is a global minimum. (This assumption is equivalent to the fact that $F$ is proper as a map of $\bbR^n$ into $\bbR$.) For an $F$ with these properties the Legendre transform 
\begin{equation}\label{diffeo}
\frac{\partial F}{\partial x}: \bbR^n \to \bbR^n
\end{equation}
is a diffeomorphism of $\bbR^n$ onto a convex open subset $\mathcal U$ of $\bbR^n$. For simplicity let's assume that this convex set is $\bbR^n$ itself. This is the case, for example, if $F$ has quadratic growth at infinity. Then the inverse of the diffeomorphism (\ref{diffeo}) is the Legendre transform
\begin{equation}
\label{LT}
\frac{\partial G}{\partial y} : \bbR^n \to \bbR^n,
\end{equation}
where for $y=\frac{\partial F}{\partial x}(x)$,
\begin{equation}
\label{GF}
G(y) = -x \cdot y +F(x).
\end{equation}
In other words, $G(y)$ is the function (\ref{GFW}) with $W(x, y)=-x \cdot y$, and the canonical transformation 
\[
\gamma: T^*\bbR^n \to T^* \bbR^n
\]
is the linear symplectomorphism
\begin{equation}
\label{lin.sym.}
\gamma(x, \xi) = (-\xi, x). 
\end{equation}
For more details about the classical Legendre transformation, see \cite{C}. 
\end{remark}

To apply theorem 3.1  to explicit examples (like the example we just described) it is important to know that the set, $\Gamma$, in (\ref{horiLag-c}), is the graph of a canonical transformation. We will discuss some necessary and sufficient conditions for this to be the case, beginning with the following simple result: 
\begin{theorem}
$\Gamma$ is the graph of a canonical transformation if and only if, for every $y \in \mathcal U_2$, the map 
\begin{equation}\label{3.8}
x \in \mathcal U_1 \mapsto \frac{\partial W}{\partial y}(x,y) \in T_y^*\mathcal U_2
\end{equation}
is a diffeomorphism onto an open subset of $T_y^*\mathcal U_2$. 
\end{theorem}
\begin{proof}
For fixed $y$ the map 
\bel{3.9}
x \in \mathcal U_1 \to (-\frac{\partial}{\partial x}W(x,y), \frac{\partial}{\partial y}W(x,y))
\ee
maps $\mathcal U$ onto the preimage  of $T_y^*\mathcal U_2$ in $\Gamma$. Hence if the map (\ref{3.8})  is a diffeomorphism of $\mathcal U_1$ onto an open subset of $T_y^*\mathcal U_2$, $\Gamma$ is the graph of the canonical transformation, $\gamma$, mapping the horizontal Lagrangian submanifold 
\bel{3.10}
\Lambda_y = \{ (x, -\frac{\partial}{\partial x}W(x,y))\ |\ x \in \mathcal U_1\}
\ee
of $T^*\mathcal U_1$ bijectively onto an open subset of $T_y^*\mathcal U_2$. Moreover, the converse is also true: If $\Gamma$ is the graph of a canonical transformation then it has to map $\Lambda_y$ diffeomorphically onto an open subset of $T_y^*\mathcal U_2$ and hence (\ref{3.8}) is a diffeomorphism of $\mathcal U_1$ onto the same open subset of $T_y^*\mathcal U_2$. 
\end{proof}

From this observation one gets immediately the following necessary condition for $\Gamma$ to be the graph of a canonical transformation: 
\begin{proposition}
If $\Gamma$ is the graph of a  canonical transformation $\gamma$, the matrix 
\begin{equation}\label{3.11}
\left[  \frac{\partial^2 W}{\partial x_i\partial y_j}\right]_{1 \le i, j \le n}
\end{equation}
is non-singular for all $(x, y) \in \mathcal U_1 \times U_2$. 
\end{proposition}
\begin{proof}
This is just the condition that for all $y$, the map (\ref{3.8}) is locally a diffeomorphism at $x \in \mathcal U_1$. 
\end{proof}

We will next describe a sufficient condition for the map (\ref{3.8}) to be a diffeomorphism onto an open subset of $T_y^*\mathcal U_2$. (This condition may seem, at first glance, to be very restrictive, but it will turn out to be satisfied in most of the examples we'll be considering below.)

\begin{theorem}
\label{TH3.5}
Suppose that $\mathcal U_1$ is convex, that (\ref{3.11}) is non-degenerate for all $(x, y)$ in $\mathcal U_1 \times \mathcal U_2$ and that $W(x,y)$ is a homogeneous quadratic polynomial in $x$. Then (\ref{3.8}) is a diffeomorphism of $\mathcal U_1$ onto an open subset of $T_y^*\mathcal U_2$. 
\end{theorem}
\begin{proof}
For each fixed $y \in \mathcal U_2$, let $f: \mathbb R^n \to \mathbb R^n$  be the map $f(x)=\frac{\partial W}{\partial y}(x,y)$. Then $f$ is  a quadratic map. As a consequence, $f$  satisfies  the following remarkable mean-value formula 
\begin{equation}\label{MVP}
f(v+w)-f(v) = df(v+\frac w2) w.
\end{equation}
To see this, one can write 
\[
f(v+tw)=f(v)+tf_1(v,w)+t^2f_2(v,w).
\]
It follows that
\[
f(v+w)-f(v) =\left. \frac{d}{dt}\right|_{t=\frac 12}f(v+tw) =df(v+\frac w2)w. 
\] 

Since the map $\frac{\partial W}{\partial y}$ is quadratic in $x$, and its $x$ differential is non-degenerate, we conclude that it is globally injective, and thus a diffeomorphism onto its image. 
\end{proof}

We will next turn to the ``horizontality" issue. Suppose $\Lambda_1$ is the horizontal Lagrangian submanifold, (\ref{horiLag-a}),  of $T^*\mathcal U_1$. What about $\Lambda_2 = \gamma(\Lambda_1)$? We will prove that the following is a sufficient condition for this to be horizontal as well.
\begin{theorem}
Suppose that for all $y \in \Lambda_2$ the function 
\bel{3.12}
x \to W(x, y)+F(x)
\ee
has a unique critical point $x_0$, and that this critical point is a nondegenerate minimum. Then $\Lambda_2$ is a horizontal Lagrangian submanifold of $T^*\mathcal U_2$ of the form (\ref{horiLag-c}), and the ``$G$" in (\ref{horiLag-b}) is the function (\ref{GFW}). 
\end{theorem}
\begin{proof}
Modulo the assumptions above the unique critical point, $x_0$, of the function (\ref{3.12}) is a  $C^\infty$ function $f(y)$ of $y$. Thus if $G(y)=W(f(y),y)+F(f(y))$, 
\bel{3.13}
\frac{\partial G}{\partial y} = \frac{\partial W}{\partial y}(f(y),y) + (\frac{\partial W}{\partial x}(x_0,y) + \frac{\partial F}{\partial x}(x_0))\frac{\partial f(y)}{\partial y}
\ee
and since $x_0$ is a critical point of (\ref{3.12}) the second summand is zero. Thus $G$ is given by (\ref{horiLag-b}).  
\end{proof}

To describe how the results above can be applied to inverse spectral problems we will discuss below a couple simple such applications, and then in section 4 we'll make some more general and systematic applications of these results. The first simple application will be a slightly improved version of the inverse spectral results we proved in \S 1 for the operator (\ref{RSch}). Namely we'll show that the assumption (\ref{DGS-Cb}) can be dropped. 

\begin{theorem}\label{bremoved}
Let $V(r_1, \cdots, r_n)$ be a $\mathbb T^n$ invariant potential function on $\mathbb R^{2n}$ satisfying (\ref{DGS-Ca}) and ({\ref{DGS-Cc}}), then $V$ is determined by the semi-classical equivariant spectrum of the Schr\"odinger operator $-\hbar^2\Delta+V$.
\end{theorem}
\begin{proof}
We apply the results above to the function  
\[ W(y, \alpha) = \frac 12 \sum_i \frac{\alpha_i^2}{y_i}. \]
It is easy to see that for any $\alpha$,
\[ 
\frac{\partial W}{\partial \alpha}: (y_1, \cdots, y_n) \to (\frac{\alpha_1}{y_1}, \cdots, \frac{\alpha_n}{y_n})
\] 
is a diffeomorphism and for any $s$, 
\[
\frac{\partial W}{\partial y}: (\alpha_1, \cdots, \alpha_n) \to -\frac 12(\frac{\alpha_1^2}{y_1^2}, \cdots, \frac{\alpha_n^2}{y_n^2})
\]
is a diffeomorphism. Moreover, for any fixed $\alpha$, $W$ is a strictly convex function of $s$. As a consequence, conditions (a) and (c) guarantee the existence and uniqueness of a minimum of the function $V(y)+W(y, \alpha)$. The conclusion follows. 
\end{proof}

Our second application will be to the operator (\ref{P}) on the toric variety $M=\mathbb{CP}^1 \times \cdots \times \mathbb{CP}^1$. 
Recall that the moment polytope  for $M$ under the standard torus action is 
\[
\mathcal P  = (-c_1, c_1) \times \cdots \times (-c_n, c_n).
\]   
\begin{theorem}\label{CP1thm}
Let $V$ be a smooth potential on $M$ that induces a function $V=V(y_1, \cdots, y_n)$ on $\mathcal P$,  satisfying the convexity condition
\begin{equation}\label{2nd}
\left[ \frac{\partial^2 V}{\partial y_i \partial y_j} \right] \ge 0.
\end{equation}

In addition, suppose that $V$ satisfies the evenness condition 
\begin{equation}
\label{even}
V(y_1, \cdots, y_n) = V(\pm y_1, \cdots, \pm y_n).
\end{equation}
on $\mathcal P$, and the monotonicity condition 
\begin{equation}\label{1st}
\frac{\partial V}{\partial y_i}<0, \qquad 1 \le i \le n
\end{equation}  
on the set $0<y_i<c_i$, then $V$ is spectrally determined  by the equivariant spectrum of the Schr\"odinger operator (\ref{P}).  
 
Alternately, suppose that for each $i$, $V$ satisfies 
\begin{equation}
\frac{\partial V}{\partial y_i}<0, \qquad \mbox{on the hyperplane \ }  y_i=0.   
\end{equation}\label{bdycd} 
Then the restriction of $V$ to the subregion 
\[R=\{y\ |\ y_i>0, i=1, \cdots, n\}\] 
is spectrally determined.

\end{theorem}
\begin{proof}[Proof of the first assertion]
The proof is almost the same as the proof described in \S 1, i.e. one can rewrite the function 
\begin{equation}
\label{min_cp1}
G(\alpha_1^2, \cdots, \alpha^2_n) = \min_y \left( V(y)+\sum \alpha_i^2 \left( \frac 1{c_i-y_i} + \frac 1{c_i+y_i}\right)\right)
\end{equation}
in the form 
\[
G(s) = \min_r (r \cdot s-F(r))) 
\]
via the change of variable 
\[y_i \to r_i= \frac 1{c_i-y_i} + \frac 1{c_i+y_i},  \]
and apply the ordinary Legender transform inversion formula. The convexidty and monotonicity conditions on $V$ are used to garantee the convexity of $F$.  

\noindent \textit{Proof of the second assertion.} In this case one argues as in the proof of theorem \ref{bremoved}, i.e. one take the function $W$ to be 
\[\sum \alpha_i^2 \left( \frac 1{c_i-y_i} + \frac 1{c_i+y_i} \right)\]
and show that the maps $\frac{\partial W}{\partial y}$ and $\frac{\partial W}{\partial \alpha}$ are diffeomorphisms (for $y$ in the region $R$). The convexicty condition (\ref{2nd}) ensures that $V(y)+W(y, \alpha)$ has a unique minimum, and the condition (\ref{bdycd}) guarantee that the minimum is in the region $R$.  
\end{proof}

\section{Inverse spectral results on toric manifolds}
 
In general the Delzant polytope of a symplectic toric manifold $M$ is given by a set of linear inequalities 
\[
l_i(x) = \sum_{k=1}^n l_i^j x_j + l_i^0 >0, \qquad i =1, \cdots, d. 
\]
In this case the function $W$ is given by 
\[
W(x,\alpha) = \frac 12\sum_{i=1}^d \frac{ \left( \sum_k l_i^k\alpha_k \right)^2}{ \sum_k l_i^k x_k + l_i^0 }.
\]
It is easy to calculate
\[
\frac{\partial W}{\partial \alpha_j} =  \sum_{i=1}^d \frac{ \sum_k l_i^k\alpha_k }{ \sum_k  l_i^k x_k + l_i^0} l_i^j  
\]
and 
\[
\frac{\partial^2 W}{\partial \alpha_j\partial x_l} =  \sum_{i=1}^d \frac{  \sum_k l_i^k\alpha_k }{ \left( \sum_k l_i^k x_k + l_i^0 \right)^2} l_i^j l_i^l.  
\]
So if we denote 
\[
L = \left[ l_j^i \right]_{ d \times n}
\]
and let $A$ be the $d \times d$ diagonal matrix $\mathrm{diag}(A_1, \cdots, A_d)$, where
\[
A_i = \frac{ -2 \sum_m l_i^m\alpha_m }{ \left( \sum_{m=1}^n l_i^m x_m + l_i^0 \right)^2} = \frac{-2 dl_i(\alpha)}{(l_i(x))^2},
\]
then we arrive at
\begin{lemma}
\[
\left[ \frac{\partial^2 W}{\partial \alpha_j\partial x_k} \right]_{n \times n}  = L^TAL.
\]
\end{lemma}


Now suppose $\mathcal R \subset \mathcal P$ is a convex region and $\mathcal S \subset \bbR^n$ a connected (conical) region, so that the matrix
\begin{equation}
\label{RS1}
 \left[ \frac{\partial^2 W}{\partial \alpha_j\partial x_k} \right] 
\end{equation} 

\noindent is non-degenerate for all $ x \in \mathcal R  $ and all $ \alpha \in \mathcal S. $

Let $\partial_1 \mathcal R$ be the subset of $ \partial \mathcal{R} $ where  $W<\infty$. At each point $x \in \partial_1 \mathcal R$ we denote by $\nu_x$  the outer normal vector to $\mathcal R$. In what follows we will assume 
\begin{equation}
\label{II}
\mbox{For each \ }x \in \partial_1 \mathcal R, \nabla_x (V+W) \cdot \nu_x>0.
\end{equation}

Finally let $V$ be a $G$-invariant potential on $M$ and thus by invariance a potential function on $\mathcal P$. We will assume 
\begin{equation}
\label{III}
V \mbox{\ is convex on\ } \mathcal P.
\end{equation}

\begin{theorem}
Under conditions (\ref{RS1}), (\ref{II}) and (\ref{III}), one can e-spectrally determine the potential $V$ on $\mathcal R$. 
\end{theorem}
\begin{proof}
First one can check that $W$ is convex in $x$. So if $V$ is strictly convex, then $V+W$ is strictly convex. It follows that $V+W$ admits a unique minimum in $\mathcal P$.  At the same time, according to (\ref{II}), the minimum of $V+W$ must sit in $\mathcal R$.  Now  apply the inversion formula for the generalized Legendre transform to $ V $ and $ W. $
\end{proof}

\subsection{Inverse spectral results on $\bbC \bbP^n$}

For $M=\bbC \bbP^n$, its moment polytope is $n$-simplex 
\[
\mathcal P: x_1>0, \cdots, x_n > 0, 1-\sum_{i=1}^n x_i > 0.   
\]

According to (\ref{Walphax}) we have
\begin{equation}\label{Walphaxcpn}
W(\alpha, x) =\frac 12 \left( \sum  \frac{\alpha_i^2}{x_i} +  \frac{(\sum \alpha_i)^2}{1-\sum x_i} \right).
\end{equation}
Let $\mathcal R \subset \mathcal P$ be the portion of $\mathcal P$ that is defined by 
\begin{equation}\label{RCPn}
\mathcal R: x_1 >0, \cdots, x_n>0, \frac 12 < \sum x_i<1. 
\end{equation}
\begin{lemma}\label{det}
The matrix 
\begin{equation}\label{Hes}
\left[ \frac{\partial^2 W}{\partial x_i \partial \alpha_j}\right]_{1 \le i, j \le n}
\end{equation}
is non-degenerate for all $x \in \mathcal R$ and all $\alpha \in \bbR^n_+$. 
\end{lemma}
\begin{proof}
By direct computation 
\[
\frac{\partial W}{\partial \alpha_j} = \frac{\alpha_j}{x_j} + \frac{\sum \alpha_i}{1- \sum x_i}
\]
and 
\[
\frac{\partial^2 W}{\partial x_i\partial \alpha_j} = -\frac{\alpha_j}{x_j^2} \delta_{ij} + \frac{\sum \alpha_i }{(1-\sum x_i)^2}.
\]
A simple induction yields 
\[
\mathrm{det}\begin{pmatrix} a_1+1 & 1 & \cdots & 1 \\
1 & a_2+1 & \cdots & 1 \\
\vdots & \vdots & \ddots & \vdots \\
1 & 1 & \cdots & a_n+1
\end{pmatrix}
= a_1 \cdots a_n (1+\sum \frac 1{a_i}).
\]
And it follows that
\[
\mathrm{det} \left[ \frac{\partial^2 W}{\partial x_i \partial \alpha_j}\right] 
= \frac{(-1)^n\alpha_1 \cdots \alpha_n \sum \alpha_i }{x_1^2\cdots x_n^2 (1-\sum x_i)^2}
\left[ \frac{(1-\sum x_i)^2}{\sum \alpha_i} - \sum \frac{x_i^2}{\alpha_i} \right].
\]
It remains to prove  	
\begin{equation}\label{ine}
 \frac{(1-\sum x_i)^2}{\sum \alpha_i} < \sum \frac{x_i^2}{\alpha_i}
\end{equation}
for all $x \in \mathcal R$ and all $\alpha \in \bbR_+^n$, but this follows from the definition of $\mathcal R$ together with the well-known Cauchy inequality: 
\[
\left(\sum \frac{x_i^2}{\alpha_i}\right)\left( \sum \alpha_i\right) \ge \left( \sum x_i \right)^2 > \left(1-\sum x_i \right)^2.  
\]
\end{proof}

\begin{remark}
One can easily see from the proof above that for any point $x$ on the hyperplane $\sum x_i=\frac 12$ and any point $\alpha = tx$,  the matrix (\ref{Hes}) is degenerate at $(x, \alpha)$. Moreover, with a little bit more work one can show that for any $x$ satisfying $\sum x_i < \frac 12$, there exists an $\alpha$  so that the matrix (\ref{Hes}) is degenerate at $(x, \alpha)$. So the region $\mathcal R$ is the maximal possible region that satisfies the lemma above. 
\end{remark}

Now we turn to the inverse spectral problem. We will assume that the potential $V$ on $\bbC\bbP^n$  is $\bbT^n$ invariant and hence can be viewed as a function on $\mathcal P$. Moreover, we will assume 
\begin{equation}\label{scon}
V \mbox{\ is strictly convex in \ }\mathcal P.
\end{equation} 
and
\begin{equation}\label{neggrad}
\mbox{ On the set\ }x_1 +\cdots+x_n=\frac 12, \frac{\partial V}{\partial x_i}<0 \mbox{\ for all\ }i.
\end{equation} 

Note that for any fixed $\alpha \in \bbR_+^n$, the function $W$ is strictly convex as a function of $x$,  and tends to infinity as $x$ tends to the boundary of $\mathcal P$. On the other hand side, as a potential on $\bbC\bbP^n$ the function $V$ is smooth up to the boundary of $\mathcal P$, and thus all its derivatives are bounded.  So the condition (\ref{scon}) ensures that the function $V(x)+W(\alpha, x)$ admits a unique minimum on $\mathcal P$. 

\begin{lemma}\label{uniqmin}
Let $V$ be a potential function defined on $\mathcal P$ satisfying (\ref{scon}) and (\ref{neggrad}). Then for any $\alpha \in \bbR_+^n$, the function $V(x)+W(\alpha, x)$ has a unique minimum in $\mathcal R$.  
\end{lemma}
\begin{proof}
We need to show that the minimum of $V(x)+W(\alpha, x)$ occurs in $\mathcal R$. Or equivalently, we only need to show that the global minimum of $V(x)+W(\alpha, x)$ on $\overline{\mathcal R}$ is not taken on the hyperplane $\sum x_i=\frac 12$. In fact, on this hyperplane we have 
\[
\frac{\partial W}{\partial x_i} = -\frac{\alpha_i^2}{x_i^2} + \frac{(\sum \alpha_i)^2}{(1-\sum x_i)^2} = -\frac{\alpha_i^2}{x_i^2} + \frac{(\sum \alpha_i)^2}{(\sum x_i)^2}. 
\]
For each fixed $\alpha \in \bbR_+^m$ and for each $x$ on this hyperplane, we just choose an index $i_0$ so that
\[
\frac{\alpha_{i_0}}{x_{i_0}} = \max \{\frac{\alpha_{i}}{x_{i}} \ |\ 1 \le i \le m \},
\]
then we must have 
\[
\frac{\partial}{\partial x_{i_0}} (V(x)+W(\alpha, x)) < 0.
\]
This implies that for any $x$ on the hyperplane $\sum x_i=1/2$,  $V(x)+W(\alpha, x)$ takes a smaller value at a point $x+\varepsilon (0, \cdots, 1, \cdots, 0)^T \in \mathcal R$, where $1$ appears at the $i_0^{th}$ position. 
\end{proof}

Now we can state the main theorem. 
\begin{theorem}
\label{TH4.6}
Let $V$ be a potential function on $\mathcal P$ that satisfies conditions (\ref{scon}) and (\ref{neggrad}). Then the equivariant spectrum of $\hbar^2 \Delta+V$ determines $V$ in $\mathcal R$. 
\end{theorem}
\begin{proof}
According to theorem \ref{toricm}, the quantity 
\[
G(\alpha) = \min_x \Big( V(x)+W(\alpha, x) \Big)
\]
is spectrally determined for all $\alpha \in \bbR_+^n$. In particular, the Lagrangian submanifold 
\[
\Lambda_2 = \{(\alpha, \frac{\partial G}{\partial \alpha})\ |\ \alpha \in \bbR_+^n\}
\] 
is spectrally determined. Let 
\[
\gamma: T^*\mathcal R \to T^*\bbR_+^n
\] 
be the canonical transform associated to the graph of $W$, then 
\[
\Lambda_1 = \gamma^{-1}(\Lambda_2)
\] 
is spectrally determined. But we also have 
\[
\Lambda_1=\{(x, \frac{\partial V}{\partial x})\ |\ x \in \mathcal R\}.
\] 
It follows that the restriction of $V$ to $\mathcal R$ is spectrally determined, up to a constant.  
\end{proof}

\subsection{Inverse spectral results on Hirzebruch surfaces}

By using similar techniques to those in the previous section, one can prove an analog of theorem \ref{TH4.6} for Hirzebruch surfaces. Let's briefly describe it in this section. Recall that a Hirzebruch surface $\mathcal H_n$ is a four dimensional symplectic toric manifold whose moment polytope is given by 
\begin{equation}\label{HirzP}
\mathcal P: x_1 > 0, x_2 > 0, 1-x_2 > 0, n+1-x_1 -nx_2 >0,
\end{equation}
where $n$ is a non-negative integer. It follows that
\begin{equation}\label{HirzW}
W(\alpha, x)=\frac 12\left[ \frac{\alpha_1^2}{x_1} + \frac{\alpha_2^2}{x_2}+\frac{\alpha_2^2}{1-x_2}+\frac{(\alpha_1+n\alpha_2)^2}{n+1-x_1-nx_2} \right].
\end{equation}
We start by studying the critical points of $W$. 
\begin{lemma}
For $ \alpha $ fixed, the critical points of $W$ lie on the curve 
\[
x_1= \frac{n+1}{2}-\frac n2 \left( x_2+ \frac{\sqrt n x_2(1-x_2)}{\sqrt{1-2x_2}}\right). 
\]
\end{lemma}
\begin{proof}
Obviously $W$ has no critical  point in the region $x_2 \ge \frac 12$. For $x_2<\frac 12$ let
\begin{equation}
\label{x2star}
x_2^* = \frac{\sqrt n x_2(1-x_2)}{\sqrt{1-2x_2}},
\end{equation}
Then the critical point equation  becomes
\[\aligned
0& =2\frac{\partial W}{\partial x_1} = -\frac{\alpha_1^2}{x_1^2}+\frac{(\alpha_1+n\alpha_2)^2}{(n+1-x_1-nx_2)^2}, \\
0& =2\frac{\partial W}{\partial x_2} = -\frac{n\alpha_2^2}{(x^*_2)^2}+ \frac{n(\alpha_1+n\alpha_2)^2}{(n+1-x_1-nx_2)^2}.
\endaligned\]
So for a point to be a critical point of $W$ for some $\alpha$, it must satisfy
\[
\frac{\alpha_1^2}{x_1^2} = \frac{\alpha_2^2}{(x^*_2)^2} = \frac{(\alpha_1+n\alpha_2)^2}{(x_1+nx_2^*)^2},
\] 
and as a consequence, 
\[
x_1+nx_2^* = n+1-x_1-nx_2.
\]
\end{proof}

Now let $ \mathcal{R} $ be the region defined by the equations $ x_1 >0, $ $ 0 < x_2 < \frac 12, $ $ n+1 -x_1 -nx_2 >0 $ and $ 2x_1 > n+1 -n (x_1 + x^{\ast}_2), $ i.e. the region depicted in the figure below
 
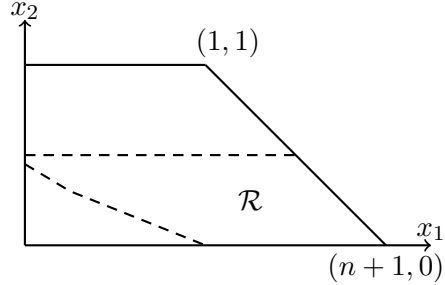
\begin{center}
\begin{tikzpicture}[scale=.6] 
\draw[thick, ->](0, 0)--(9, 0);
\draw[thick, ->](0, 0)--(0, 5);
\draw[thick, smooth](0, 4)--(4,4);
\draw[thick, smooth](4, 4)--(8,0);
\draw[thick, dash pattern=on 4pt off 3pt](0, 2)--(6,2); 
\draw[thick, dash pattern=on 4pt off 3pt]plot coordinates {(0,1.8)(1,1.2)(2,0.8)(3,0.4)(4,0)}; 
\node at (5, 1) {$\mathcal R$};
\node at (4.5, 4.5) {$(1,1)$};
\node at (8, -.5) {$(n+1, 0)$};
\node at (9, .3) {$x_1$};
\node at (0, 5.2) {$x_2$};
\end{tikzpicture} 
\end{center}

As in the case of $\bbC\bbP^m$, one can prove
\begin{lemma}
Let $W(\alpha, x)$ be given by (\ref{HirzW}). Then
\begin{enumerate}
\item For  $x \in \mathcal R$ and  $\alpha \in \bbR_+^2$, the matrix $\left[ \frac{\partial^2W}{\partial x_i\partial \alpha_j}\right]$ is non-degenerate.
\item For  fixed $x \in \mathcal R$, the map $\frac{\partial W}{\partial x}$ is a diffeomorphism from $\bbR^2_+$ onto its image. 
\end{enumerate}
\end{lemma}
\begin{proof}
Note that in region $\mathcal R$ we have 
\[
x_1+nx_2^* > n+1-x_1-nx_2.
\]
So 
\[\aligned
\det\left[\frac{\partial^2W}{\partial x_i\partial \alpha_j}\right] & = \det \begin{pmatrix}
\frac{-\alpha_1}{x_1^2}+\frac{\alpha_1+n\alpha_2}{(n+1-x_1-nx_2)^2} & \frac{n(\alpha_1+n\alpha_2)}{(n+1-x_1-nx_2)^2} \\
\frac{n(\alpha_1+n\alpha_2)}{(n+1-x_1-nx_2)^2}  & \frac{-n\alpha_2}{(x_2^*)^2}+\frac{n^2(\alpha_1+n\alpha_2)}{(n+1-x_1-nx_2)^2} 
\end{pmatrix}
\\
&=  \frac{n\alpha_1\alpha_2}{x_1^2(x_2^*)^2} - \left(\frac{\alpha_1}{x_1^2}+\frac{\alpha_2}{n(x_2^*)^2} \right)\frac{n^2(\alpha_1+n\alpha_2)}{(n+1-x_1-nx_2)^2} \\
& <   \frac{n\alpha_1\alpha_2}{x_1^2(x_2^*)^2} - \left(\frac{\alpha_1}{x_1^2}+\frac{\alpha_2}{n(x_2^*)^2} \right)\frac{n^2(\alpha_1+n\alpha_2)}{(x_1+nx_2^*)^2} \\
& = \frac{n^2\alpha_1(n\alpha_2)(\alpha_1+n\alpha_2)}{x_1^2(nx_2^*)^2(x_1+nx_2^*)^2}\left[\frac{(x_1+nx_2^*)^2}{\alpha_1+n\alpha_2}-\frac{(nx_2^*)^2}{n\alpha_2}-\frac{x_1^2}{\alpha_1}  \right]
\\
&<0, 
\endaligned\]
where again the last inequality follows from the Cauchy-Schwartz inequality. 

The proof of (2) is similar to the proof of theorem \ref{TH3.5}. 
\end{proof}

Finally let $V$ be a  potential so that 
\begin{equation}\label{sconH}
V \mbox{\ is strictly convex in \ }\mathcal P.
\end{equation} 
and
\begin{equation}\label{neggradH}
\aligned
& \mbox{ On the line\ }x_2=\frac 12, \quad \frac{\partial V}{\partial x_2}>0; \\
& \mbox{ On the curve\ }x_1 =\frac{n+1}{2} - \frac n2(x_2+x_2^*), \quad \frac{\partial V}{\partial x_i}<0 \mbox{\ for all\ }i.
\endaligned\end{equation} 
\begin{lemma}
Suppose $V$ is a function satisfying (\ref{sconH}) and (\ref{neggradH}), then for any $\alpha \in \bbR_+^2$, the function $V(x)+W(\alpha, x)$ admits a unique minimum on $\mathcal R$. 
\end{lemma}
\begin{proof}
We only need to analyze the function $V(x)+W(\alpha, x)$ along the line $x_2=\frac 12$ and along the curve $x_1 =\frac{n+1}{2} - \frac n2(x_2+x_2^*)$. 

In the first case, we notice that along $x_2=\frac 12$, 
\[
\frac{\partial W}{\partial x_2} = \frac{n(\alpha_1+n\alpha_2)^2}{(n+1-x_1-nx_2)^2}>0,
\]
so the condition $\frac{\partial V}{\partial x_2}>0$ implies that the function $V(x)+W(\alpha, x)$ takes a smaller value in $\mathcal R$. 

In the second case, one can argue exactly as in the proof of lemma \ref{uniqmin} to conclude that either $\frac{\partial W}{\partial x_1}<0$ or $\frac{\partial W}{\partial x_2}<0$. 
\end{proof}

Given these lemmas, the following theorem is now obvious:
\begin{theorem}
If $V$ is a $\mathbb T^2$-invariant potential on the Hirzebruch surface $\mathcal H_n$ whose restriction on $\mathcal R$ satisfies (\ref{sconH}) and (\ref{neggradH}), then the equivariant spectrum of the operator $\hbar^2 \Delta +V$ determines $V$ in the region $\mathcal R$. 
\end{theorem}

\begin{remark}
With slightly more work, one can prove a similar theorem on toric varieties whose moment polytope $\mathcal P$ is bounded by  coordinate hyperplanes 
\[
x_i>0 \quad (1 \le i \le i_0)  
\]
and pairs of parallel coordinate hyperplanes
\[
x_i>0, \quad c_i-x_i>0 \quad  (i_0+1 \le i \le n)
\]
 together with one exceptional ``skew hyperplane''
\[
b_0-(b_1x_1 + \cdots + b_n x_n) >0, 
\] 
where the $b_i$'s are nonnegative integers. 
\end{remark}

\section{Local inverse and spectral rigidity results}

We'll begin this section with a quick review of the material in \S 1: Let $M$ be a $2n$-dimensional Riemannian manifold, $G$ an $n$ dimensional torus, $\tau: G \times M \to M$ an isometric action of $G$ on $M$, $V: M \to \mathbb R$ a $G$-invariant $C^\infty$ function and $\hbar^2 \Delta_M+V$ the Schr\"odinger operator (\ref{P}). As in section 1 we will denote by $M_0$ the open subset of $M$ on which $G$ acts freely and by $X$ the quotient $M_0/G$ From the symbol of the operator (\ref{P}) one gets a reduced symbol
\begin{equation}\label{reducedsym}
p_\alpha(x, \xi)=|\xi|_x^2+V_\alpha(x),
\end{equation}
on $T^*X$, where $V_\alpha(x)=W(x, \alpha)+V(x)$ is the function (\ref{V}). 

As we pointed out in \S 1, all of the results of this paper are basically corollaries of the formula (\ref{mualphah}) which asserts that the push-forward $(p_\alpha)_*\nu$ of the symplectic volume form on $T^*X$ is a spectral invariant of the operator (\ref{P}), and, in particular, that the points on the real line where this measure fails to be smooth are spectral invariants. However, the only use of this we've made so far is to conclude that if $(p_\alpha)_*\nu$ is supported on the interval $[c_\alpha, \infty)$ then $c_\alpha$ is a spectral invariant. Hence its not unreasonable to hope that the other points on the interval $[c_\alpha, \infty)$ where $(p_\alpha)_*\nu$ is singular might have inverse spectral applications, and our goal in this section will be to explore this possibility. 

The singularities of $(p_\alpha)_*\nu$ are critical values of the symbol $p_\alpha$, and by (\ref{reducedsym}) these critical values have to occur at points where $\xi=0$, and hence coincide with critical values of $V_\alpha(x)$. However, its not clear that all critical values of $V_\alpha(x)$ are spectral invariants since one critical value can correspond to several critical points and the contribution to $(p_\alpha)_*\nu$ coming from these points can cancel out. Let's suppose however that for $\alpha=\alpha_0$ there is  a unique critical point, $p_0$, with critical value $c_0=V_{\alpha_0}(p_0)$ and that this point is a non-degenerate local minimum. We will prove 

\begin{theorem}
If the function $W(x, \alpha)$ satisfies the non-degeneracy condition (\ref{3.11}) at $(x_0, \xi_0)$, then $V$ is e-spectrally determined on a neighborhood of $x_0$. 
\end{theorem}
\begin{proof}
Let $\xi_0=-\frac{\partial W}{\partial x}(x_0, \xi_0)$ and $\eta_0=\frac{\partial W}{\partial \alpha}(x_0, \alpha_0)$. Then if $W$ satisfies the non-degeneracy condition (\ref{3.11}) at $(x_0, \alpha_0)$, $W$ is the generating function of a canonical transformation, $\gamma$, mapping a neighborhood of $(x_0, \xi_0)$ in $T^*X$ onto a neighborhood of $(\alpha_0, \eta_0)$ in $T^*\mathbb R^n$. Moreover if $V(x, \alpha)$ has a non-degenerate minimum at $(x_0, \alpha_0)$, then there exists a neighborhood $U_0$ of $x_0$ in $X$ and a neighborhood $Q$ of $\alpha_0$ in $\mathbb R^n$ such that for all $\alpha \in Q$,  $V_\alpha(x)$ has a unique critical point of $V_{\alpha_0}(x)$ on $X$ at which $V_{\alpha_0}$ takes the critical value $c_0$. The same is true of the critical point $x_\alpha$ of $V_\alpha(x)$. Hence if we let $G(\alpha)=V_\alpha(x_\alpha)$ the function $G: Q \to \mathbb R$ is e-spectrally determined and hence, by theorem 3.5, the function $V|_{\mathcal U_0}$ is e-spectrally determined.
\end{proof}
 
If there are several critical points at which $V_{\alpha_0}(x)$ takes the critical value $c_0$, we can no longer conclude that $V$ is e-spectrally determined on a neighborhood of $x_0$ but we can prover the following local ``spectrally rigidity" result: 
\begin{theorem}
Given a $G$-invariant $C^\infty$ function, $V_1: M \to \mathbb R$, there exists a neighborhood $\mathcal U_0$ of $x_0$ and an $\varepsilon_0>0$ such that for all $0 < \varepsilon < \varepsilon_0$, the equivariant spectrum of the operator $\hbar^2 \Delta +V +\varepsilon V_1$ determines $V_1| {\mathcal U_0}$. 
\end{theorem}
 The proof of this theorem is practically identical with the proof of theorem 5.1 and will be omitted 
 
 To apply these results to concrete examples we need to know that the condition (\ref{3.11}) is satisfied, and as we saw in section 4 this is not trivial to verify even in simple examples. We will show however that for an arbitrary toric variety this condition is satisfied if we make a fairly restrictive assumption about the location of the point, $x_0$. Recall that for a toric variety the manifold $X=M_0/G$ can be identified with the interior of the moment polytope $\mathcal P$, i.e. one has an identification $x \in X \mapsto s \in \mathcal P$. 
\begin{theorem}
The condition (\ref{3.11}) is satisfied at $(s_0, \alpha_0)$ if $s_0$ lies sufficiently close to a vertex of $\Delta$ and sufficiently far away from all the other vertices and if, in addition, 
\bel{5.1}
\langle n_i, \alpha_0\rangle \ne 0, i=1, \cdots, d,
\ee
where the $n_i$'s are the normal vectors to the facets of $\mathcal P$ meeting at this vertex. 
\end{theorem} 
\begin{proof}
We can without loss of generality assume that the vertex above is the origin and that the facets meeting at this vertex are the hyperplanes, $s_i=0$. Then by (\ref{Walphax}), 
\[
W(s, \alpha) = \frac 12 \sum \frac{\alpha_i^2}{s_i^2} + \cdots
\]
and 
\[
\frac{\partial W}{\partial \alpha_i \partial s_j} = -2 \delta_{ij}\frac{\alpha_i}{s_i^2} + \cdots,
\]
where the ``$\cdots$" are negligible small compared with the first term. 
\end{proof}
 
Thus since the vertices of $\mathcal P$ are the images in $M/G$ of the fixed points of $G$ one obtain from theorem 5.1 and 5.2 interesting local inverse spectral and spectral rigidity results for the Schr\"odinger operator (\ref{P}) on small neighborhood of these fixed points.

\section{Reduction in stages}

Let $X$ be a compact $n$-dimensional manifold and $M \to X$ a circle bundle over $X$. Then the space of functions on $M$ which transform under the action, $\tau: M \times S^1 \to M$, of $S^1$ on $M$ according to the law
\begin{equation}
\tau_\theta^* f = e^{i\theta}f
\end{equation}
can be viewed as sections of a line bundle, $\mathbb L \to X$, and the spaces of functions which transform according to the law 
\begin{equation}\label{Ntheta}
\tau_\theta^* f = e^{iN\theta}f
\end{equation}
as sections of its $N$th tensor power, $\mathbb L^N$. Moreover, if $M$ is equipped with an $S^1$-invariant Riemannian metric and $\Delta: C^\infty(M) \to C^\infty(M)$ is its associated Laplace operator, then $\Delta$ preserves the space of functions (\ref{Ntheta}) and hence induces on $\mathbb L^N$ a Laplace operator 
\begin{equation}\label{LaplaceN}
\Delta_N: C^\infty(\mathbb L^N) \to C^\infty(\mathbb L^N),
\end{equation}
and if one sets $\hbar = \frac 1N$ and let $\hbar$ tends to zero, then the self adjoint operator, $\hbar^2 \Delta_N$, and the one parameter group of unitary operators that it generates become the quantization of the Hamiltonian system on $T^*X$ obtained by symplectic reduction from geodesic flow on $T^*M$. In particular the generator of this system is the Hamiltonian 
\begin{equation}\label{Hxxi}
H(x,\xi)=\langle \xi, \xi \rangle_x^*+V(x),
\end{equation}
where $\langle \cdot, \cdot \rangle_m$ is the Riemannian inner product on $T_mM$, $\langle \cdot ,\cdot \rangle_x$, $x=\pi(m)$, the ``reduced" Riemannian inner product on $T_xX$ and $V(x) = \langle \frac{\partial}{\partial \theta}, \frac{\partial}{\partial \theta} \rangle_m$. Moreover, if $\mathcal U$ is an open subset of $X$ over which $M$ admits a trivialization
\begin{equation}
M|_{\mathcal U} = \mathcal U \times S^1,
\end{equation}
the operator $\hbar^2 \Delta_N$, $N=\frac 1\hbar$, becomes a semi-classical Schr\"odinger operator of the form (\ref{P}) with leading symbol (\ref{Hxxi}).

Suppose now that the manifold $M$ admits additional symmetries, i.e. an $n$-torus, $G$, of isometries. Then by ``reduction in stages" the spectral invariants of this Schr\"odinger operator that one gets by reduction with respect to semi-classical weights, $\frac{\alpha}{\hbar}$, of $G$ are also invariants of $\Delta_N$. In other words the inverse spectral techniques described in \S 1-\S 5 apply not just to semi-classical Schr\"odinger operators, but also, by reduction in stages, to Laplace operators on line bundles.  

An application of this observation is the following: Suppose one is given a Riemannian manifold, $X$, a circle bundle
\begin{equation}
\pi: M \to X
\end{equation}
and a connection on this bundle, i.e. an $S^1$ equivariant splitting
\begin{equation}\label{TMSplit}
TM=T_{vert}+\pi^*TX. 
\end{equation}
Then one gets a Riemannian metric on $M$ by requiring this splitting to be an orthogonal splitting and by requiring the metric on the second summand to be the pull back of the metric on $TX$. (Thus the only ambiguity in the definition of this metric is the inner product on the first factor, i.e. the function $V$ in the paragraph above.) Moreover, having equipped $M$ with such a metric one gets a Laplace operator on $M$ and, by restriction, a Laplace operator on each of the line bundles (\ref{LaplaceN}), and a natural question to ask ( a question that can, in particular, be answered using the techniques of this paper) is whether $V$ is a  spectral invariant of the operators (\ref{LaplaceN}). In other words, given a metric on $M$ that is compatible with the connection (\ref{TMSplit}), is it spectrally determined?

\end{document}